\documentclass[11pt]{amsart}
\usepackage{amsmath,amssymb,amsbsy,amsfonts,latexsym,amsthm,amscd,amsxtra,amsrefs}
\usepackage[T1]{fontenc}
\usepackage{graphics,latexsym,enumerate}

\newtheorem{lemma}{Lemma}
\newtheorem{thm}{Theorem}
\newtheorem{prop}{Proposition}

\begin{document}

\title[An improved Strichartz estimate for systems]{An improved Strichartz estimate for systems with divergence free data}
\author{Sagun CHANILLO}
  \address{Department of Mathematics \\ Rutgers University \\ NJ 08854}
  \email{chanillo@math.rutgers.edu}

\author{Po-Lam YUNG}
  \address{Department of Mathematics \\ Rutgers University \\ NJ 08854}
  \email{pyung@math.rutgers.edu}
  
\thanks{The first author is supported by NSF grant DMS-0855541. The second author would like to thank Jonathan Luk for a number of inspiring discussions.}

\begin{abstract}
Using the div-curl inequalities of Bourgain-Brezis \cite{MR2057026} and van Schaftingen \cite{MR2078071}, we prove an improved Strichartz estimate for systems of inhomogeneous wave and Schrodinger equations, for which the inhomogeneity is a divergence-free vector field at each given time. The novelty of the result is that one can allow $L^1_x$ norms of the inhomogeneity in the right hand side of the estimate.
\end{abstract}

\subjclass[2000]{35L05, 35L10}

\maketitle

In this paper we are interested in improved Strichartz estimates for systems of inhomogeneous wave and Schrodinger equations, when the inhomogeneity is a divergence free vector field at any given time. The starting point is the following simple observation:

\begin{prop} \label{prop1}
Suppose $u \colon \mathbb{R}^{1+2} \to \mathbb{R}^2$ is a (weak) solution of the following system of wave equations
$$
\begin{cases}
\square u = f \\
\left. u \right|_{t = 0} = u_0 \\
\left. \partial_t u \right|_{t = 0} = u_1
\end{cases}
$$
where $f = (f_1, f_2) \colon \mathbb{R}^{1+2} \to \mathbb{R}^2$ is a divergence free vector field at each given time $t$, i.e. $$\partial_{x_1} f_1 + \partial_{x_2} f_2 = 0$$ for each $t$. Then
$$
\|u\|_{C^0_t L^2_x} + \|\partial_t u\|_{C^0_t \dot{H}^{-1}_x} \leq C \left( \|u_0\|_{L^2} + \|u_1\|_{\dot{H}^{-1}} + \|f\|_{L^1_t L^1_x} \right).
$$
\end{prop}

Here $\square = -\partial_t^2 + \Delta$ is the d'Alembertian acting componentwise on $u$, and $\dot{H}^s$ is the homogeneous Sobolev space $\dot{W}^{s,2}$.

A remarkable feature in our estimate is that on the right hand side we only need the $L^1_x$ norm of $f$, which is usually not possible in the classical energy (or Strichartz) inequalities. Our estimate is only possible because we have the additional structural assumption that $f$ is a divergence free vector field at each time $t$. In fact if one tries to prove the Proposition using Sobolev embedding naively without using this divergence free assumption, say when $u_0 = u_1 = 0$, then one would estimate, at any time $t$,
\begin{align*}
\|u\|_{L^2_x}
&= \left\| \int_0^t \frac{\sin((t-s)\sqrt{-\Delta})}{\sqrt{-\Delta}} f(s,x) ds \right\|_{L^2_x} \\
&\leq \int_0^t \left\| \frac{1}{\sqrt{-\Delta}} f(s,x) \right\|_{L^2_x} ds \\
&\leq \int_0^t \left\| R_1 f(s,x) \right\|_{L^1_x} + \left\| R_2 f(s,x) \right\|_{L^1_x} ds
\end{align*}
where $R_j$ are the Riesz transforms on $\mathbb{R}^2$, which are unfortunately not bounded on $L^1$.

Before we state our more general results, we first give a short proof of Proposition \ref{prop1}. The proof relies on the following simple observation that we first learned from Bourgain-Brezis \cite{MR2057026}:

\begin{lemma}[Bourgain-Brezis] \label{lemma1}
For each divergence free vector field $F = (F_1, F_2)$ on $\mathbb{R}^2$ with $F \in L^1$, there exists $G \in L^2$ such that $F_1 = \partial_{x_2} G$ and $F_2 = -\partial_{x_1} G$ with $\|G\|_{L^2} \leq C \|F\|_{L^1}$.
\end{lemma}

\begin{proof}
The assumption that $\text{div} F = 0$ allows one to find $G$ such that $F_1 = \partial_{x_2} G$ and $F_2 = -\partial_{x_1} G$, and $G \in L^2$ by Sobolev embedding because $\nabla G = (-F_2, F_1) \in L^1$.
\end{proof}

In fact in \cite{MR2057026} and the subsequent work \cite{MR2293957}, \cite{MR2078071}, Bourgain-Brezis and van Schaftingen obtained some far-reaching generalizations of this simple lemma, and the latter is what we shall exploit in our more general result in this paper.

\begin{proof}[Proof of Proposition \ref{prop1}]
Let $f$ be as in the Proposition. Applying the lemma to $f(t,\cdot)$ at each time $t$, we obtain a function $g(t,\cdot)$ such that $f_1 = \partial_{x_2} g$, $f_2 = -\partial_{x_1} g$, and $\|g\|_{L^2_x} \leq C \|f\|_{L^1_x}$ at each time $t$.
Now the classical energy estimate says that
$$
\|u\|_{C^0_t L^2_x} + \|\partial_t u\|_{C^0_t \dot{H}^{-1}_x} \leq C \left(\|u_0\|_{L^2} + \|u_1\|_{\dot{H}^{-1}} + \|(-\Delta)^{-\frac{1}{2}}f\|_{L^1_t L^2_x} \right).
$$
Since for each fixed $t$,
$$
\|(-\Delta)^{-\frac{1}{2}}f\|_{L^2_x} = \|(-\Delta)^{-\frac{1}{2}}\nabla g\|_{L^2_x} \leq C \|g\|_{L^2_x} \leq \|\nabla f\|_{L^1_x}
$$
by Sobolev embedding, our result follows.
\end{proof}

The key observation in proving Proposition~\ref{prop1} is that the coefficients of $f$ are in $\dot{H}^{-1}(\mathbb{R}^2)$ for all $t$ under the given conditions. We remark that there are other situations under which the inhomogeneity of the wave equation lies in $\dot{H}^{-1}(\mathbb{R}^2)$; one instance is given in the appendix.

In what follows, we derive improved Strichartz inequalities similar to Proposition~\ref{prop1}, using generalizations of Lemma~\ref{lemma1} by van Schaftingen.

\section{Strichartz estimates for the wave equation}

In the sequel we shall consider vector fields $f \colon \mathbb{R}^{1+n} \to \mathbb{R}^n$. Our main result for the wave equation is the following:

\begin{thm} \label{thm:homoStr}
Suppose $n \geq 2$, and let $u \colon \mathbb{R}^{1+n} \to \mathbb{R}^n$ be a (weak) solution of the system
$$
\begin{cases}
\square u = f \\
\left. u \right|_{t = 0} = u_0 \\
\left. \partial_t u \right|_{t = 0} = u_1
\end{cases}
$$
where $f(t,x) \colon \mathbb{R}^{1+n} \to \mathbb{R}^n$ is a divergence free vector field for all $t$. Suppose $s, k \in \mathbb{R}$, $2 \leq q, \tilde{q}  \leq \infty$, $2 \leq r < \infty$, and we assume further that $\tilde{q} > \frac{4}{n-1}$ if $n = 2$ or 3. Suppose $(q,r)$ satisfies the wave admissibility condition
$$\frac{1}{q} + \frac{n-1}{2r} \leq \frac{n-1}{4},$$
and the following scale invariance condition is verified:
$$\frac{1}{q} + \frac{n}{r} = \frac{n}{2} - s = \frac{1}{\tilde{q}'} + n - 2 - k.$$
Then
\begin{align*}
& \|u\|_{L^q_t L^r_x} + \|u\|_{C^0_t \dot{H}^s_x} + \|\partial_t u\|_{C^0_t \dot{H}^{s-1}_x} \\
\leq  & C \left( \|u_0\|_{\dot{H}^s} + \|u_1\|_{\dot{H}^{s-1}} + \|(-\Delta)^{\frac{k}{2}} f\|_{L^{\tilde{q}'}_t L^1_x} \right).
\end{align*}
\end{thm}

To prove this, the starting point is the following result of van Schaftingen \cite{MR2078071}:

\begin{thm}[van Schaftingen] \label{thm:vssp}
Let $F \colon \mathbb{R}^n \to \mathbb{R}^n$ be a divergence free vector field with components in $L^1$. Then for any $0 < \alpha < n$,
$$\|F\|_{\dot{W}^{-\alpha, \frac{n}{n-\alpha}}} \leq C \|F\|_{L^1}.$$
\end{thm}

The Theorem was stated in \cite{MR2078071} only for $0 < \alpha \leq 1$, but the rest of the theorem follows easily from Sobolev embedding of $\dot{W}^{-\alpha, \frac{n}{n-\alpha}}$ into $\dot{W}^{-\beta, \frac{n}{n-\beta}}$ in $\mathbb{R}^n$ if $0 < \alpha \leq \beta < n$.

We also need the following version of the Strichartz estimate for
the scalar equation. It is stated in Proposition 3.1 of Ginibre-Velo
\cite{MR1351643} for the non-endpoint case (where both $(q,r),
(\tilde{q},\tilde{r}) \ne (2, \frac{2(n-1)}{n-3})$), and the
endpoint case can be proved using the technology of Keel-Tao
\cite{MR1646048}.
\begin{lemma} \label{thm:generalStr}
Suppose $n \geq 2$, and let $u \colon \mathbb{R}^{1+n} \to \mathbb{R}$ be a (weak) solution of
$$
\begin{cases}
\square u = h \\
\left. u \right|_{t = 0} = u_0 \\
\left. \partial_t u \right|_{t = 0} = u_1
\end{cases}
$$
Suppose $s, \gamma \in \mathbb{R}$, $2 \leq q, \tilde{q}  \leq \infty$, $2 \leq r, \tilde{r} < \infty$, $(q,r)$ and $(\tilde{q},\tilde{r})$ satisfy the wave admissibility conditions
$$\frac{1}{q} + \frac{n-1}{2r} \leq \frac{n-1}{4}, \quad \frac{1}{\tilde{q}} + \frac{n-1}{2\tilde{r}} \leq \frac{n-1}{4},$$
and the following scale invariance condition is verified:
$$\frac{1}{q} + \frac{n}{r} = \frac{n}{2} - s = \frac{1}{\tilde{q}'} + \frac{n}{\tilde{r}'} - 2 - \gamma.$$
Then
\begin{align*}
&\|u\|_{L^q_t L^r_x} + \|u\|_{C^0_t \dot{H}^s_x} + \|\partial_t u\|_{C^0_t \dot{H}^{s-1}_x} \\
\leq & C \left( \|u_0\|_{\dot{H}^s} + \|u_1\|_{\dot{H}^{s-1}} + \|(-\Delta)^{\frac{\gamma}{2}} h\|_{L^{\tilde{q}'}_t L^{\tilde{r}'}_x} \right).
\end{align*}
\end{lemma}

For the convenience of the reader, we pause to outline a proof of the end-point case of Lemma~\ref{thm:generalStr}:

\begin{proof}[Proof of the end-point case of Lemma~\ref{thm:generalStr}] 
The desired estimate of $\|u\|_{C^0_t \dot{H}^s_x} +
\|\partial_t u\|_{C^0_t \dot{H}^{s-1}_x}$ follows from the statement
of Corollary 1.3 of \cite{MR1646048}. To prove
$$\|u\|_{L^q_t L^r_x} \leq C \|(-\Delta)^{\frac{\gamma}{2}}
h\|_{L^{\tilde{q}'}_t L^{\tilde{r}'}_x},$$ one observes that since
$2 \leq q, \tilde{q} \leq \infty$, $2 \leq r, \tilde{r} < \infty$,
one can restrict attention to the situation where the frequency
support of $h(t,\cdot)$ is contained in an annulus of size $2^j$ by
using the Littlewood-Paley square function. By scale invariance we
can take $j = 0$. In that case $(-\Delta)^{\frac{\gamma}{2}
 }$ on the right hand side can be dropped, and the result follows from Theorem 1.2 of \cite{MR1646048}.
\end{proof}

Theorem~\ref{thm:homoStr} can be seen as the limiting case of
Lemma~\ref{thm:generalStr} when $\tilde{r} = \infty$ except when
$(n,\tilde{q},\tilde{r})=(2,4,\infty)$ or $(3,2,\infty)$. It says
one still has the Strichartz inequality if in addition $f$ is a
vector field at each time $t$, and $f(t,x)$ is divergence free for
all $t$.

\begin{proof}[Proof of Theorem \ref{thm:homoStr}] Assume $n$, $q$, $\tilde{q}$, $r$, $k$ and $s$ be as given in the statement of the Theorem. Then when $n \geq 4$, from $2 \leq \tilde{q} \leq \infty$ one automatically has $$\frac{n}{2} - \frac{2n}{(n-1)\tilde{q}} > 0,$$ and the same inequality holds when $n = 2$ or $3$ because then we assumed $\tilde{q} > \frac{4}{n-1}$. As a result, one can pick some $\alpha \in (0, \frac{n}{2} - \frac{2n}{(n-1)\tilde{q}}]$. Now let $\tilde{r} = \frac{n}{\alpha}$, and $\gamma = k - \alpha$. Then $\tilde{r} < \infty$, $\frac{1}{\tilde{q}} + \frac{n-1}{2\tilde{r}} \leq \frac{n-1}{4}$, which in particular implies that $\tilde{r} \geq 2$. The scale invariance condition in Lemma~\ref{thm:generalStr} is also verified. Hence
\begin{align*}
&\|u\|_{L^q_t L^r_x} + \|u\|_{C^0_t \dot{H}^s_x} + \|\partial_t u\|_{C^0_t \dot{H}^{s-1}_x} \\
\leq & C \left( \|u_0\|_{\dot{H}^s} + \|u_1\|_{\dot{H}^{s-1}} + \|(-\Delta)^{\frac{k-\alpha}{2}} f\|_{L^{\tilde{q}'}_t L^{\tilde{r}'}_x} \right).
\end{align*}
Now invoking Theorem~\ref{thm:vssp} and the divergence free condition on $f$ for each time $t$, we get $$\|(-\Delta)^{\frac{k-\alpha}{2}} f\|_{L^{\tilde{r}'}_x} \leq C \|(-\Delta)^{\frac{k}{2}} f\|_{L^1_x},$$ from which the desired inequality follows. Note this is possible because $\alpha \in (0,n)$ automatically by our choice of $\alpha$.
\end{proof}

We remark that under the conditions of Theorem~\ref{thm:homoStr}, we necessarily have $s \geq 0$, and when $n \geq 3$ we necessarily have $k > 0$. In fact $k \geq \frac{n-3}{2}$ when $n \geq 3$, and $k = 0$ is impossible when $n = 3$ because we assumed that $\tilde{q} > 2$ when $n = 3$.

We also remark that in Theorem~\ref{thm:homoStr}, when the initial conditions $u_0$ and $u_1$ are zero, one can actually obtain a wider range of exponents for which the desired inequality holds. This can be thought of as a limiting case of an inhomogeneous Strichartz estimate of Taggart \cite{Taggart}, whose origin goes back to the work of Foschi \cite{MR2134950}. To illustrate this, we state the following Theorem in 3 space dimensions.

\begin{thm} \label{thm:inhomowave}
Suppose $n = 3$, and let $u \colon \mathbb{R}^{1+3} \to \mathbb{R}^3$ be a (weak) solution of the system
$$
\begin{cases}
\square u = f \\
\left. u \right|_{t = 0} = 0 \\
\left. \partial_t u \right|_{t = 0} = 0
\end{cases}
$$
where $f(t,x) \colon \mathbb{R}^{1+3} \to \mathbb{R}^3$ is a divergence free vector field for all $t$. Suppose $k \in \mathbb{R}$, $1 < q, \tilde{q}  \leq \infty$, $2 \leq r < \infty$, and
$$
\frac{1}{q} + \frac{1}{\tilde{q}} < \min \left\{1, \frac{k+1}{2}\right\}.
$$
Suppose further that $(q,r)$ satisfies the wave acceptability condition
$$\frac{1}{q} + \frac{2}{r} < 1 \quad \text{or} \quad (q,r) = (\infty,2),$$
and that the following scale invariance condition is verified:
$$\frac{1}{q} + \frac{3}{r} = 2 - k - \frac{1}{\tilde{q}}.$$
Then
$$
\|u\|_{L^q_t L^r_x} \leq  C \|(-\Delta)^{\frac{k}{2}} f\|_{L^{\tilde{q}'}_t L^1_x}.
$$
\end{thm}

To prove this, we need the following scalar inhomogeneous Strichartz
estimate, which is a consequence of Corollary 8.7 of Taggart
\cite{Taggart} in 3 space dimensions:

\begin{thm}[Taggart] \label{thm:Taggart}
Suppose $n = 3$, and let $u \colon \mathbb{R}^{1+3} \to \mathbb{R}$ be a (weak) solution of
$$
\begin{cases}
\square u = h \\
\left. u \right|_{t = 0} = 0 \\
\left. \partial_t u \right|_{t = 0} = 0.
\end{cases}
$$
Suppose $\gamma \in \mathbb{R}$, $1 < q, \tilde{q}  \leq \infty$, $2 \leq r, \tilde{r} < \infty$,
$$
\frac{1}{q} + \frac{1}{\tilde{q}} < 1, \quad \text{and} \quad \frac{1}{q} + \frac{1}{\tilde{q}} \leq \frac{\gamma+1}{2}.
$$
Suppose further that the exponents satisfy the wave acceptability condition
$$\frac{1}{q} + \frac{2}{r} < 1 \quad \text{or} \quad (q,r) = (\infty,2),$$
$$\frac{1}{\tilde{q}} + \frac{2}{\tilde{r}} < 1 \quad \text{or} \quad (\tilde{q},\tilde{r}) = (\infty,2),$$
and that the following scale invariance condition is verified:
$$\frac{1}{q} + \frac{3}{r} = 2 - \gamma - \frac{1}{\tilde{q}} - \frac{3}{\tilde{r}}.$$
Then
$$
\|u\|_{L^q_t L^r_x} \leq  C \|(-\Delta)^{\frac{\gamma}{2}} h\|_{L^{\tilde{q}'}_t L^{\tilde{r}'}_x}.
$$
\end{thm}

\begin{proof}[Proof of Theorem~\ref{thm:Taggart}] 
Under the conditions of Theorem~\ref{thm:Taggart}, one has
$$\frac{1}{q} + \frac{2}{r} < 1 \quad \text{or} \quad (q,r) =
(\infty,2),$$
$$\frac{1}{\tilde{q}} + \frac{2}{\tilde{r}} < 1 \quad \text{or} \quad (\tilde{q},\tilde{r}) = (\infty,2),$$ and $$\frac{1}{r} + \frac{1}{\tilde{r}} \leq 1 - \frac{1}{q} - \frac{1}{\tilde{q}},$$ the last inequality following from the condition $\frac{1}{q} + \frac{1}{\tilde{q}} \leq \frac{\gamma+1}{2}$ and the scale invariance condition. Thus one can find $r_1 \leq r$, $\tilde{r_1} \leq \tilde{r}$ such that the wave acceptability conditions
$$\frac{1}{q} + \frac{2}{r_1} < 1 \quad \text{or} \quad (q,r_1) = (\infty,2)$$
and
$$\frac{1}{\tilde{q}} + \frac{2}{\tilde{r_1}} < 1 \quad \text{or} \quad (\tilde{q},\tilde{r_1}) = (\infty,2),$$
are satisfied, with
$$
\frac{1}{r_1} + \frac{1}{\tilde{r_1}} = 1 - \frac{1}{q} -
\frac{1}{\tilde{q}}.
$$
Clearly $r_1, \tilde{r_1} \in [2,\infty)$. As a result, Corollary
8.7 of Taggart \cite{Taggart} applies, yielding
Theorem~\ref{thm:Taggart}.
\end{proof}

\begin{proof}[Proof of Theorem~\ref{thm:inhomowave}]
Assume $q$, $\tilde{q}$, $r$ and $k$ be as given in the statement of the Theorem. Then since $$\frac{1}{q} + \frac{1}{\tilde{q}} < \frac{k+1}{2} \quad \text{and} \quad \frac{1}{\tilde{q}} < 1,$$ one can pick a small $\alpha > 0$ such that $$\frac{1}{q} + \frac{1}{\tilde{q}} \leq \frac{(k-\alpha)+1}{2} \quad \text{and} \quad \frac{1}{\tilde{q}} + \frac{2 \alpha}{3} < 1.$$ Now let $\tilde{r} = \frac{3}{\alpha}$, and $\gamma = k - \alpha$. Then $\tilde{r} < \infty$, $\frac{1}{q} + \frac{1}{\tilde{q}} \leq \frac{\gamma+1}{2}$, $\frac{1}{\tilde{q}} + \frac{2}{\tilde{r}} < 1$, which in particular implies that $\tilde{r} > 2$. The scale invariance condition in Theorem~\ref{thm:Taggart} is also verified. Hence
$$
\|u\|_{L^q_t L^r_x} \leq C \|(-\Delta)^{\frac{k-\alpha}{2}} f\|_{L^{\tilde{q}'}_t L^{\tilde{r}'}_x}.
$$
Now invoking Theorem~\ref{thm:vssp} and the divergence free condition on $f$ for each time $t$, we get $$\|(-\Delta)^{\frac{k-\alpha}{2}} f\|_{L^{\tilde{r}'}_x} \leq C \|(-\Delta)^{\frac{k}{2}} f\|_{L^1_x},$$ from which the desired inequality follows. Note this is possible because $\alpha \in (0,3)$ automatically by our choice of $\alpha$; in fact $\alpha < \frac{3}{2}$ since $\frac{1}{\tilde{q}} + \frac{2 \alpha}{3} < 1$.
\end{proof}

\section{Strichartz estimates for the Schrodinger equation}

Again, we consider vector fields $f \colon \mathbb{R}^{1+n} \to \mathbb{R}^n$. The main result is the following.

\begin{thm} \label{thm:Schrod}
Suppose $n \geq 2$, and $u \colon \mathbb{R}^{1+n} \to \mathbb{R}^n$ is a (weak) solution of the system of Schrodinger equations
$$
\begin{cases}
i \partial_t u + \Delta u = f \\
\left. u \right|_{t=0} = u_0,
\end{cases}
$$
where $f(t,x) \colon \mathbb{R}^{1+n} \to \mathbb{R}^n$ is a divergence free vector field for all $t$. Suppose $2 \leq q, \tilde{q} \leq \infty$, $2 \leq r < \infty$, $s \geq 0$, $k > s$, and the following scale invariance conditions are satisfied:
$$
\frac{2}{q} + \frac{n}{r} = \frac{n}{2} - s, \qquad \frac{2}{\tilde{q}} = \frac{n}{2} - k + s.
$$
Then
$$
\|u\|_{C^0_t \dot{H}^s_x} + \|u\|_{L^q_t L^r_x} \leq C \left( \|u_0\|_{\dot{H}^s} + \|(-\Delta)^{\frac{k}{2}} f\|_{L^{\tilde{q}'}_t L^1_x}\right).
$$
\end{thm}

In its proof we need Theorem~\ref{thm:vssp} in the previous Section, as well as the following Strichartz inequality for the scalar Schrodinger equation (which follows from Corollary 1.4 of Keel-Tao \cite{MR1646048} and the Sobolev inequality):

\begin{lemma}
Suppose $n \geq 2$, and $u \colon \mathbb{R}^{1+n} \to \mathbb{R}$ is a (weak) solution of
$$
\begin{cases}
i \partial_t u + \Delta u = h \\
\left. u \right|_{t=0} = u_0.
\end{cases}
$$
Suppose $2 \leq q, \tilde{q} \leq \infty$, $2 \leq r, \tilde{r} < \infty$, $s \geq 0$, $\gamma > s$, and the following scale invariance conditions are satisfied:
$$
\frac{2}{q} + \frac{n}{r} = \frac{n}{2} - s, \qquad \frac{2}{\tilde{q}} + \frac{n}{\tilde{r}} = \frac{n}{2} - \gamma + s.
$$
Then
$$
\|u\|_{C^0_t \dot{H}^s_x} + \|u\|_{L^q_t L^r_x} \leq C \left( \|u_0\|_{\dot{H}^s} + \|(-\Delta)^{\frac{\gamma}{2}} f\|_{L^{\tilde{q}'}_t L^{\tilde{r}'}_x}\right).
$$
\end{lemma}

Theorem~\ref{thm:Schrod} can be thought of as the limiting case of the above Lemma when $\tilde{r} = \infty$, which only works because we assumed that the inhomogeneity $f(t,x)$ is a divergence free vector field at each time $t$.

\begin{proof}[Proof of Theorem~\ref{thm:Schrod}]
Assume $n$, $q$, $\tilde{q}$, $r$, $k$ and $s$ be as given in the statement of the Theorem. Then $k - s > 0,$ one can pick some $\alpha \in (0, \min\{k-s,\frac{n}{2}\}]$. Now let $\tilde{r} = \frac{n}{\alpha}$, and $\gamma = k - \alpha$. Then $2 \leq \tilde{r} < \infty$, $\gamma > s$, and $$\frac{2}{\tilde{q}} + \frac{n}{\tilde{r}} = \frac{n}{2} - \gamma + s.$$
Hence
$$
\|u\|_{C^0_t \dot{H}^s_x} + \|u\|_{L^q_t L^r_x} \leq C \left( \|u_0\|_{\dot{H}^s} + \|(-\Delta)^{\frac{k-\alpha}{2}} f\|_{L^{\tilde{q}'}_t L^{\tilde{r}'}_x}\right).
$$
Now we invoke Theorem~\ref{thm:vssp} and the divergence free condition on $f$ at each time $t$; this is possible because $\alpha \in (0,n)$ automatically by our choice of $\alpha$. Thus we get $$\|(-\Delta)^{\frac{k-\alpha}{2}} f\|_{L^{\tilde{r}'}_x} \leq C \|(-\Delta)^{\frac{k}{2}} f\|_{L^1_x},$$ from which the desired inequality follows.
\end{proof}

\section{Appendix}

In this appendix we prove another improved Strichartz inequality for the wave equation in $\mathbb{R}^{1+2}$. Here we only need to work with scalar equations.

\begin{prop}
Suppose $u \colon \mathbb{R}^{1+2} \to \mathbb{R}$ satisfies
$$
\begin{cases}
\square u = \det(\nabla_x F) \\
\left. u \right|_{t = 0} = u_0 \\
\left. \partial_t u \right|_{t = 0} = u_1
\end{cases}
$$
where $F$ is a map from $\mathbb{R}^{1+2}$ to $\mathbb{R}^2$, and $\det(\nabla_x F)$ denotes its Jacobian determinant in the $x$ variable. Then
$$
\|u\|_{C^0_t L^2_x} + \|\partial_t u\|_{C^0_t \dot{H}^{-1}_x} \leq C \left( \|u_0\|_{L^2} + \|u_1\|_{\dot{H}^{-1}} +  \int \|\nabla_x F\|_{L^2_x}^2 dt \right).
$$
\end{prop}

It is clear that $\|\nabla_x F\|_{L^2_x}^2$ controls the $L^1_x$ norm of $\det(\nabla_x F)$, but unfortunately this is not enough if one wants to prove the Proposition. On the other hand, we claim
$$
\|(-\Delta)^{-\frac{1}{2}} \det(\nabla_x F)\|_{L^2_x} \leq C \|\nabla_x F\|_{L^2_x}^2.
$$
This follows from Wente's inequality; see e.g. Theorem 0.2 of Chanillo-Li \cite{MR1190215}. Alternatively, since we are in 2 space dimensions, by compensation compactness (see Coifman-Lions-Meyer-Semmes \cite{MR1225511}), $\|\nabla_x F\|_{L^2_x}^2$ controls the Hardy $\mathcal{H}^1_x$ norm of $\det(\nabla_x F)$, which in turn controls the negative Sobolev $\dot{H}^{-1}_x$ norm of $\det(\nabla_x F)$, from which our claim follows. Arguing using the classical energy estimate as in the proof of Proposition~\ref{prop1}, the desired estimate follows.

\bibliography{div_curl_paper}

\end{document}